\documentclass[11pt,twoside]{amsart}

\usepackage{amsmath}
\usepackage{amsthm}
\usepackage{amsfonts, amssymb}
\usepackage{mathrsfs}
\usepackage[all]{xy}
\usepackage{url}
\usepackage{graphics}
\usepackage{epstopdf}

\usepackage{latexsym}
\usepackage[dvips]{graphicx}

\theoremstyle{plain}
\newtheorem{lema}{Lemma}[section]

\newtheorem{teo}[lema]{Theorem}

\newtheorem*{intro1}{Theorem \ref{cororegular}}
\newtheorem*{intro2}{Theorem \ref{main5}}
\newtheorem{coro}[lema]{Corollary}
\theoremstyle{remark}

\newtheorem{obs}[lema]{Remark}

\theoremstyle{definition}

\newtheorem{ej}[lema]{Example}

\newcommand{\kker}{\textrm{ker}}
\newcommand{\Deck}{\textrm{Deck}}

\newcommand{\ee}{\textbf{E}}
\newcommand{\hh}{\mathcal{H}}
\newcommand{\kp}{\mathcal{K}}
\newcommand{\x}{\mathcal{X}}
\newcommand{\degg}{\textrm{h}}

\def\Z{\mathbb{Z}}

\def\set#1{\{#1\}}
\begin{document}

\title[The second homotopy group in terms of colorings]{The second homotopy group in terms of colorings of locally finite models and new results on asphericity}

\author[J.A. Barmak]{Jonathan Ariel Barmak}
\author[E.G. Minian]{Elias Gabriel Minian}

\address{Departamento  de Matem\'atica--IMAS\\
 FCEyN, Universidad de Buenos Aires\\ Buenos
Aires, Argentina}

\email{jbarmak@dm.uba.ar}
\email{gminian@dm.uba.ar}

\begin{abstract}
We describe the second homotopy group of any CW-complex $K$ by analyzing the universal cover of a locally finite model of $K$ using the notion of $G$-coloring of a partially ordered set. As applications we prove a generalization of the Hurewicz theorem, which relates the homotopy and homology of non-necessarily simply-connected complexes, and derive new results on asphericity for two-dimensional complexes and group presentations.
\end{abstract}

\subjclass[2010]{55Q05, 57M10, 57M20,  06A06, 55U10}

\keywords{Covering maps, second homotopy group, asphericity, posets, finite topological spaces.}

\maketitle

\section{Introduction} \label{sectionintro}

Every CW-complex has a \it locally finite model. \rm This is a classical result of McCord \cite[Theorem 3]{Mcc} who considered for any regular CW-complex $K$ the space $\x (K)$ of cells of $K$ with some specific topology, and defined a weak homotopy equivalence $\mu :K \to \x (K)$. The space $\x (K)$ can be viewed as a poset. The interaction between the topological and combinatorial nature of $\x(K)$ allows one to develop new techniques to attack problems of homotopy theory of CW-complexes (see \cite{Bar2}). 


In this paper we use locally finite models to describe the second homotopy group of CW-complexes. The notion of $G$-coloring of a poset allows us to classify all the regular coverings of the space $\x (K)$. In particular, we obtain a description of the  universal cover $E$ of $\x(K)$ which is used to find an expression for the boundary map of a chain complex whose homology coincides with the singular homology of $E$. By the Hurewicz theorem and McCord's result, $\pi_2 (K)=H_2(E)$. One of the applications of our description is the following result which reduces to the classical Hurewicz theorem when the complex is simply-connected.

\begin{intro1} 
Let $K$ be a connected regular CW-complex of dimension $2$ and let $K'$ be its barycentric subdivision. Consider the full (one-dimensional) subcomplex $L$ of $K'$ spanned by the barycenters of the $1$-cells and $2$-cells. If the inclusion of each component of $L$ in $K'$ induces the trivial morphism between the fundamental groups, then $\pi_2 (K)=\Z [\pi_1 (K)]\otimes H_2(K)$.  
\end{intro1}

We also obtain results on asphericity of $2$-complexes and group presentations. Recall that a connected $2$-complex $K$ is aspherical if $\pi_2(K)=0$.
 
\begin{intro2} 
Let $K$ be a $2$-dimensional regular CW-complex and let $K'$ be its barycentric subdivision. Consider the full (one-dimensional) subcomplex $L\subseteq K'$ spanned by the barycenters of the $2$-cells of $K$ and the barycenters of the $1$-cells which are faces of exactly two $2$-cells. Suppose that for every connected component $M$ of $L$, $i_*(\pi_1 (M))\leq \pi_1 (K')$ contains an element of infinite order, where $i_*:\pi_1 (M) \to \pi_1 (K')$ is the map
induced by the inclusion. Then $K$ is aspherical.  
\end{intro2}

From this result one can deduce for example, the well-known fact that all compact surfaces different from $S^2$ and $\mathbb{R} P^2$ are aspherical. 

To put our results in perspective, one should recall that it is an open problem, originally posted by Whitehead, whether any subcomplex of an aspherical $2$-dimensional CW-complex is itself aspherical . We refer the reader to \cite{Bo,Ho,Wh} for more details on Whitehead's asphericity question.

In Theorem \ref{main6} we prove a result on asphericity of group presentations which resembles the homological description of $\pi_2$ by Reidemeister chains (see \cite{Bo,Si}).

\section{Colorings and a description of the second homotopy group}

A poset $X$ will be identified with a topological space with the same underlying set as $X$ and topology generated by the basis $\{U_x\}_{x\in X}$, where $U_x=\{y\in X \ | \ y\le x\}$. If $X$ and $Y$ are posets, it is easy to see that a map $X\to Y$ is continuous if and only if it is order preserving. We denote by $\kp (X)$ the simplicial complex whose simplices are the finite chains of the poset $X$ (i.e. the classifying space of $X$). A result of McCord \cite[Theorem 2]{Mcc} shows that there is a natural weak homotopy equivalence $\kp (X)\to X$. Given a regular CW-complex $K$, its \textit{face poset} $\x (K)$ is the poset of cells of $K$ ordered by the face relation. Note that the classifying space of the poset $\x (K)$ is the barycentric subdivision of $K$ and therefore, there is a weak homotopy equivalence $\mu :K \to \x (K)$. In particular the homology groups of the poset $\x (K)$ coincide with those of $K$. If $X=\x (K)$ is the face poset of a regular CW-complex, we can compute its homology by computing the cellular homology of $K$ in the standard way (see \cite[IX, \S 7]{Mas}), namely, for each $n\ge 0$ let $C_n(X)$ be the free $\mathbb{Z}$-module generated by the points $x \in X$ of height $\degg (x)=n$. Recall that $\degg (x)$ is one less than the maximum number of points in a chain with maximum $x$. Choose for each edge $(y,x)$ in the Hasse diagram of $X$ a number $[x:y]\in \{1,-1\}$ in such a way that for every $x\in X$ of height $1$, $$\sum\limits_{y\prec x} [x:y] =0,$$ and for each pair $x,z\in X$ with $\degg(x)=\degg(z)+2$, $$\sum\limits_{z\prec y \prec x} [x:y][y:z]  =0.$$

Here $y\prec x$ means that $y<x$ and there is no $y<z<x$. The differential $d:C_n(X)\to C_{n-1}(X)$ is defined by $d(x)=\sum\limits_{y\prec x} [x:y] y$ in each basic element $x$. The homology of this chain complex is then the singular homology of the poset $X$ (viewed as a topological space). The number $[x:y]$ is the incidence of the cell $y$ in the cell $x$ of $K$ for certain orientations.

Suppose $p:E\to X=\x(K)$ is a topological covering ($\x (K)$ considered as a topological space). Using that $p$ is a local homeomorphism it is easy to prove that $E$ is also the space associated to a poset. Moreover, for each $x\in E$, $p|_{U_x}: U_x\to U_{p(x)}$ is a homeomorphism. In particular $E$ is the face poset of a regular CW-complex (\cite[Proposition 4.7.23]{Bjo}). If $y\prec x$ in $E$, $p(y)\prec p(x)$ in $X$. Given a choice of the incidences in $X$, we define $[x:y]=[p(x):p(y)]$, which is a coherent choice for the incidences in $E$. Let $p:E\to \x(K)$ be a regular covering and let $G$ be its group of deck (covering) transformations. Since $G$ acts freely on $E$ and transitively on each fiber, $C_n(E)=\Z G \otimes C_n(X)$ is a free $\Z G$-module with basis $\{ x \in X \ | \ \degg(x)=n \}$. The differential $d:C_n(E)\to C_n(E)$ is a homomorphism of $\Z G$-modules. 

In \cite{paper1} we characterized the regular coverings of locally finite posets (i.e.  posets with finite $U_x$, for each $x$) in terms of \textit{colorings}. We recall this result as it will be required in the description of the universal cover.

Let $X$ be a locally finite poset. We denote by $\ee (X)$ the set of edges in the Hasse diagram of $X$. An \textit{edge-path} in $X$ is a sequence $\xi =(x_0,x_1)(x_1,x_2) \ldots (x_{k-1},x_k)$ of edges, or opposites of edges. The set of closed edge-paths from a point $x_0 \in X$, with certain identifications and the operation given by concatenation, is a group $\hh (X,x_0)$ naturally isomorphic to $\pi _1(X,x_0)$ (see \cite{BM1} and \cite{paper1}). This construction resembles the definition of the edge-path group of a simplicial complex. Given a group $G$, a \textit{$G$-coloring} of a locally finite poset $X$ is a map $c$ which assigns a \textit{color} $c(y,x)\in G$ to each edge $(y,x)\in \ee (X)$. Given a $G$-coloring, if $y\prec x$ we define $c(x,y)=c(y,x)^{-1}\in G$. A $G$-coloring of $X$ induces a \textit{weight} map which maps an edge-path $\xi =(x_0,x_1)(x_1,x_2) \ldots (x_{k-1},x_k)$ to $w_c(\xi)=c(x_0,x_1)c(x_1,x_2)\ldots c(x_{k-1},x_k)$. A $G$-coloring $c$ is said to be \textit{admissible} if for any two chains $x=x_1\prec x_2\prec \ldots \prec x_k=y,$ $x=x_1'\prec x_2'\prec \ldots \prec x_l'=y$ with same origin and same end, the weights of the edge-paths induced by the chains coincide. An admissible $G$-coloring $c$ induces a homomorphism $W_c: \hh (X,x_0) \to G$ which maps the class of a closed edge-path to its weight. The coloring $c$ is \textit{connected} if $W_c$ is an epimorphism.

Two $G$-colorings $c$ and $c'$ of $X$ are \textit{equivalent} if there exists an automorphism $\varphi :G\to G$ and an element $g_x\in G$ for each $x\in X$ such that $c'(x,y)=\varphi (g_x c(x,y) g_y^{-1})$ for each $(x,y)\in \ee (X)$.

\begin{teo}(\cite[Corollary 3.5]{paper1})
Let $X$ be a connected locally finite poset and let $G$ be a group. There exists a correspondence between the set of equivalence classes of regular coverings $p:E\to X$ of $X$ with $\Deck(p)$ isomorphic to $G$ and the set of equivalence classes of admissible connected $G$-colorings of $X$. 
\end{teo}

Here $\Deck(p)$ denotes the group of deck transformations of $p$. The covering associated to an admissible connected $G$-coloring $c$ is the covering that corresponds to the subgroup $\ker (W_c)$ of $\hh (X,x_0) \cong \pi_1 (X.x_0)$. Theorem 3.6 of  \cite{paper1} tells us explicitely how to construct the covering $E(c)$ corresponding to $c$. It is the poset $E(c)= \{(x,g) \ | \ x\in X, \ g\in G \}$ with the relations $(x,g)\prec (y, gc(x,y))$ whenever $x\prec y$ in $X$. The covering map being the projection onto the first coordinate. The group $G$ acts on $E(c)$ by left multiplication in the second coordinate.

Now, let $K$ be a regular CW-complex and suppose $c$ is any $G$-coloring of $X=\x (K)$ which corresponds to the universal cover, that is, $c$ is an admissible and connected $G$-coloring such that $E=E(c)$ is simply-connected or, equivalently, $W_c:\hh (X,x_0)\to G$ is an isomorphism. The second homotopy group of $K$ is $\pi_2(K)=\pi_2(\x (K))=H_2(E)$. The homology of $E$ can be computed using the chain complex described above. In the case that $K$ is two-dimensional, this computation is easier. In this case $E$ is a poset of height two and $C_3(E)=0$.  A chain $\alpha \in C_2(E)=\Z G \otimes C_2(X)$ is a finite sum of the form $$\alpha = \sum\limits_{\degg(x)=2} \sum\limits_{g \in G} n_g^xgx $$ where $n_g^x\in \Z$. The isomorphism between $C_2(E)$ and $\Z G \otimes C_2(X)$ identifies $(x,1)\in E$ with $x$. Thus, $d:\Z G \otimes C_2(X) \to \Z G \otimes C_2(X)$ maps $x$ to $$\sum\limits_{y\prec x} [x:y](y, c(y,x)^{-1})= \sum\limits_{y\prec x} [x:y]c(y,x)^{-1}y \in C_1(E)=\Z G \otimes C_1(X)$$ and then $d(\alpha)=\sum\limits_{\degg(x)=2} \sum\limits_{g \in G} \sum\limits_{y\prec x} [x:y]n_g^xgc(y,x)^{-1}y.$

Therefore $\pi_2 (K)=\kker (d)$ has the following description

$$\pi_2(K)=\set{ \sum\limits_{\degg(x)=2} \sum\limits_{g \in G} n_g^xgx \ | \ \sum\limits_{x \succ y} [x:y] n_{h \cdot c(y,x)}^x=0 \ \forall \ y\in X \textrm{ with } \degg (y)=1 \textrm{ and } \forall \ h\in G }.$$

On the other hand, Theorem 4.4 and Remark 4.6 in \cite{paper1} provide a concrete way to describe a coloring $\hat{c}$ which corresponds to the universal cover. Let $X$ be a locally finite poset and let $D$ be a subdiagram(=subgraph) of the Hasse diagram of $X$. Suppose that the poset which corresponds to $D$ is simply-connected and that $D$ contains all the points of $X$ (for instance, a spanning tree). Let $G$ be the group generated by the edges $e\in \ee (X)$ which are not in $D$ with the following relations. For each pair of chains $$x=x_1\prec x_2\prec \ldots \prec x_k=y, $$ $$x=x_1'\prec x_2'\prec \ldots \prec x_l'=y $$ with same origin and same end, we put a relation $$\prod\limits_{(x_i,x_{i+1})\notin D} (x_i,x_{i+1}) = \prod\limits_{(x_i',x_{i+1}')\notin D} (x_i',x_{i+1}').$$

According to Theorem 4.4 in \cite{paper1} $G$ is isomorphic to $\pi _1(X)$. Moreover, let $\hat{c}$ be the $G$-coloring defined by $\hat{c}(e)=\overline{e}$, the class of $e$ in $G$, for each $e\in \ee (X)$. If $e\in D$, $\overline{e}=1\in  G$. Then $W_{\hat{c}}:\hh (X,x_0)\to G$ is an isomorphism, so $\hat{c}$ corresponds to the universal cover of $X$. This coloring can be used in the formula above to compute $\pi_2 (K)$.

\begin{ej} \label{}
Consider the regular CW-complex $K$ in Figure \ref{proyect}. It has three $0$-cells, $a,b,c$, six $1$-cells, $q,r,s,t,u,v$ and four $2$-cells, $w,x,y,z$. The Hasse diagram of $\x (K)$ appears in Figure \ref{proyect2}. Let $D$ be the subdiagram of the Hasse diagram given by the solid edges. It is easy to see that the space corresponding to $D$ is simply connected because it is a contractible finite space(=dismantlable poset) \cite[Section 4]{Sto}. The group $G$ generated by the dotted edges $e_1,e_2,e_3,e_4, e_5$ with relations $e_4e_1=1, e_1=e_5, e_2=e_3, e_2=e_5, e_3=e_4$ is then isomorphic to the fundamental group of $K$. Hence $\pi_1(K)=\Z _2$.  

\begin{figure}[h] 
\begin{center}
\includegraphics[scale=0.6]{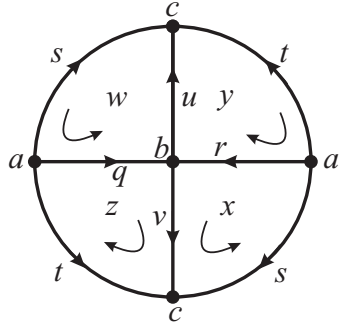}
\caption{A regular CW-structure of the projective plane.}\label{proyect}
\end{center}
\end{figure}

\begin{figure}[h] 
\begin{center}
\includegraphics[scale=0.6]{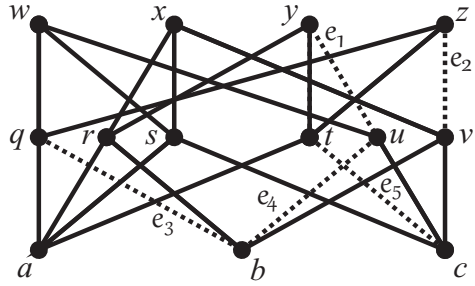}
\caption{The coloring $\hat{c}$ of $\x(K)$.}\label{proyect2}
\end{center}
\end{figure}

For each $h\in G$ and each point of $\x (K)$ of height $1$ we have one equation. These twelve equations describe $\pi _2(K)$. Denote $\gamma$ the generator of $G$. Then $\overline{e}_1=\overline{e}_2=\gamma$. We can choose the incidences $[p_0:p_1]$ according to the orientations of the cells in Figure \ref{proyect}. For instance, the equation corresponding to $u$ and $h\in G$ is $0=[w:u]n_{h \cdot \overline{uw}}^w+[y:u]n_{h\cdot \overline{e}_1}^y=n_h^w+n_{h\gamma}^y$. For each $h\in G$ the equations are

\medskip

\begin{math}
\begin{array}{lrr}
\textrm{Equations for } q:& n_{h}^w+n_h^z & =0\\
\textrm{Equations for } r:& n_{h}^x+n_h^y &=0\\
\textrm{Equations for } s:& -n_h^w-n_h^x &=0\\
\textrm{Equations for } t:& -n_{h}^y-n_h^z &=0\\
\textrm{Equations for } u:& n_{h}^w+n_{h\gamma}^y &=0\\
\textrm{Equations for } v:& n_{h}^x+n_{h\gamma}^z &=0\\
\end{array}
\end{math}

\medskip

Therefore $\pi _2(K)=\{n(w-\gamma w-x+\gamma x+y-\gamma y-z+\gamma z) \ | \ n\in \Z \}$ is isomorphic to $\Z$. In fact $K$ is just the real projective plane, so the results are not surprising. However, the example shows how to carry on the computation of $\pi _2$ for arbitrary regular CW-complexes.
\end{ej}

\begin{teo} \label{cororegular}
Let $K$ be a connected regular CW-complex of dimension $2$ and let $K'$ be its barycentric subdivision. Consider the full (one-dimensional) subcomplex $L$ of $K'$ spanned by the barycenters of the $1$-cells and $2$-cells. If the inclusion of each component of $L$ in $K'$ induces the trivial morphism between the fundamental groups, then $\pi_2 (K)=\Z [\pi_1 (K)]\otimes H_2(K)$.  
\end{teo}
\begin{proof}
Let $D$ be the subgraph of the Hasse diagram of $\x (K)$ induced by the points of height $1$ and $2$. Then $L$ is the classifying space $\kp (D)$ of the poset associated to $D$ and the weak homotopy equivalence $\mu :K' \to \x (K)$ restricts to a weak equivalence $L\to D$. Moreover, for each component $L_i$ of $L$, $\mu | _{L_i}$ is a weak equivalence between $L_i$ and a component $D_i$ of $D$. If $L_i \hookrightarrow K'$ induces the trivial map in $\pi_1$, then so does the inclusion $D_i\hookrightarrow \x (K)$. By \cite[Remark 4.2]{paper1} each admissible $G$-coloring of $\x (K)$ is equivalent to another which is trivial in the edges of $D$. In particular, if $c$ is a $\pi _1 (K)$-coloring which corresponds to the universal cover, then it is equivalent to a coloring $c'$ such that $c'(x,y)=1$ for each $(x,y)\in D$. Then $\widetilde{X}=E(c')$ is also the universal cover of $\x (K)$ and 
$$d=1\otimes \delta:C_2(\widetilde{X})=\Z [\pi_1 (K)]\otimes C_2(\x (K))\to C_{1}(\widetilde{X})=\Z [\pi_1(K)] \otimes C_{1}(\x (K)),$$
where $\delta:C_2(\x (K))\to C_{1}(\x (K))$ is the boundary map of the chain complex associated to $\x (K)$. Since $\Z [\pi _1(K)]$ is a free $\Z$-module, $\pi_2 (K)=H_2(\widetilde{X})=\Z [\pi_1 (K)] \otimes H_2(K)$ by the K\"unneth formula. 
\end{proof}

When $K$ is simply-connected, the previous result reduces to the Hurewicz Theorem for dimension $2$.

Theorem \ref{cororegular} can be restated as follows: If every closed edge-path of $K'$ containing no vertex of $K$ is equivalent to the trivial edge-path, then $\pi_2 (K)=\Z [\pi_1 (K)]\otimes H_2(K)$.

There is an obvious generalization of Theorem \ref{cororegular} to connected regular CW-complexes with no restriction on the dimension.

\begin{coro}\label{anydimensionalhurew}
Let $K$ be a connected regular CW-complex. If every closed edge-path of $K'$ containing only vertices which are barycenters of $1$, $2$ or $3$-dimensional simplices is equivalent to the trivial edge-path, then $\pi_2 (K)=\Z [\pi_1 (K)]\otimes H_2(K)$.
\end{coro}



The following is another application of our methods (compare with \cite{Whg}).

\begin{teo}
Let $X$ and $Y$ be two connected $CW$-complexes. If $Y$ is simply-connected, then $\pi _2(X\vee Y)=\pi_2 (X)\oplus (\Z [\pi_1(X)]\otimes \pi_2 (Y))$.
\end{teo}
\begin{proof}
Since each $CW$-complex is homotopy equivalent to a simplicial complex, it suffices to prove the result for face posets $X$ and $Y$ of regular CW-complexes. Here, $X\vee Y$ denotes the space whose Hasse diagram is obtained from the diagrams of $X$ and of $Y$ by identifying a minimal element of each.
Let $c$ be a coloring of $X\vee Y$ corresponding to the universal cover. Then $c$ is a $G$-coloring with $G\simeq \pi_1 (X\vee Y)\simeq \pi_1 (X)$. Since $Y$ is simply-connected, there is an equivalent $G$-coloring $c'$ which is trivial in $Y$ (once again by Lemma 4.1 or Remark 4.2 in \cite{paper1}). The restriction of $c'$ to $X$ is an admissible connected $G$-coloring. Moreover, if a closed edge-path in $X$ is in $\kker (W_{c'|_X})$, then it is in $\kker (W_{c'})=0$. Thus, it is trivial in $\hh (X\vee Y)$ and then in $\hh (X)$, since the inclusion $X\hookrightarrow X\vee Y$ induces an isomorphism between the fundamental groups. Therefore, $c'|_X$ corresponds to the universal cover of $X$.

Let $\widetilde{X\vee Y}=E(c')$ be the universal cover of $X\vee Y$. Note that $$C_n(\widetilde{X\vee Y})=\Z [G] \otimes C_n(X\vee Y)= (\Z [G] \otimes C_n(X)) \oplus (\Z [G] \otimes C_n(Y))$$ for $n=1,2$. Since $c'|_X$ corresponds to the universal cover of $X$ and $c'|_Y$ is trivial, the differential $$d: (\Z [G] \otimes C_2(X)) \oplus (\Z [G] \otimes C_2(Y)) \to (\Z [G] \otimes C_1(X)) \oplus (\Z [G] \otimes C_1(Y))$$ has the form $d=d_{\widetilde{X}} \oplus (1_{\Z [G]} \otimes d_Y)$, where $d_{\widetilde{X}}:C_2(\widetilde{X}) \to C_1(\widetilde {X})$ is the differential in the chain complex associated to the universal cover of $X$ and $d_Y:C_2(Y) \to C_1(Y)$ is the differential in the complex associated to $Y$. By the K\"unneth formula, $\pi_2 (X\vee Y)=\kker (d) =H_2 (\widetilde{X}) \oplus (\Z [G] \otimes H_2(Y))=\pi _2(X)\oplus (\Z [G] \otimes \pi _2(Y))$.   
\end{proof}

\section{Results on asphericity}

We use the methods developed above to study asphericity of two-dimensional complexes and group presentations.

\begin{teo} \label{main5}
Let $K$ be a $2$-dimensional regular CW-complex and let $K'$ be its barycentric subdivision. Consider the full (one-dimensional) subcomplex $L\subseteq K'$ spanned by the barycenters $b(\tau)$ of the $2$-cells $\tau$ of $K$ and the barycenters of the $1$-cells which are faces of exactly two $2$-cells. Suppose that for every connected component $M$ of $L$, $i_*(\pi_1 (M))\leq \pi_1 (K')$ contains an element of infinite order, where $i_*:\pi_1 (M) \to \pi_1 (K')$ is the map
induced by the inclusion. Then $K$ is aspherical.  
\end{teo}
\begin{proof}
Let $c$ be a $G$-coloring of $\x (K)$ which corresponds to the universal cover. We will use the equations describing $\pi_2 (K)$ to show that if $\alpha =\sum\limits_{\degg(x)=2} \sum\limits_{g \in G} n_g^xgx \in \pi _2(K)$, then $n_g^x=0$ for every $g\in G$ and every $x$ with $\degg (x)=2$. Let $x=\tau$ be a maximal element of $\x (K)$. Then $W=W_{c}:\hh (\x (K), x) \to G$ is an isomorphism. 

Let $Y$ be the subspace of $\x (K)$ consisting of the $2$-cells and the $1$-cells which are faces of exactly two $2$-cells. Note that $L=\kp (Y)$, so there is a weak homotopy equivalence $L\to Y$. Since $i_*(\pi_1 (L ,b(\tau)))$ contains an element of infinite order and $W$ is an isomorphism, there is a closed edge-path $\xi$ at $x$ in $Y$ of weight
$w(\xi)\in G$ of infinite order. We may assume that $\xi$ is an edge-path of minimum length satisfying this property. Suppose $\xi$ is the edge-path $x=x_0 \succ w_0 \prec x_1 \succ w_1 \prec \ldots \succ w_{k-1} \prec x_{k}=x$. 
By the minimality of $\xi$, $x_{i+1}\neq x_i$ for every $0\leq i< k$. Since $x_i$ and $x_{i+1}$ are the unique two elements covering $w_i$, the equation corresponding to $w_i$ and an element $g\in G$ is 
$$[x_i: w_i] n_{gc(w_i, x_i)}^{x_i}+[x_{i+1}:w_i]n_{gc(w_i, x_{i+1})}^{x_{i+1}}=0.$$
In particular, given $g\in G$, if $n_g^{x_i}\neq 0$, then $n_{gc(w_i,x_i)^{-1}c(w_i,x_{i+1})}^{x_{i+1}}\neq 0$.

Let $h\in G$. Suppose that $n_h^x\neq 0$. Applying the previous assertion $k$ times we obtain that $n_{hw(\xi)}^x\neq 0$. Repeating this reasoning we deduce that $n_{hw(\xi)^l}^x\neq 0$ for every $l\ge 0$. However, $w(\xi)\in G$ has infinite order and this contradicts the fact that only finitely many $n_g^z$ can be non-zero.
\end{proof}

Note that from the previous result one deduces the well-known fact that all compact surfaces different from $S^2$ and $\mathbb{R}P^2$ are aspherical. Any triangulation $K$ of such surfaces satisfies the hypotheses of the theorem since every edge of $K$ is face of exactly two $2$-simplices and the links of the vertices are connected.

 
\begin{ej}
The pinched two-handled torus and the wedge of two torii (Figure \ref{pinch}) are aspherical by Theorem \ref{main5}.
\begin{figure}[h] 
\begin{minipage}{6cm}
\vspace{.3cm}
\center{\includegraphics[scale=0.3]{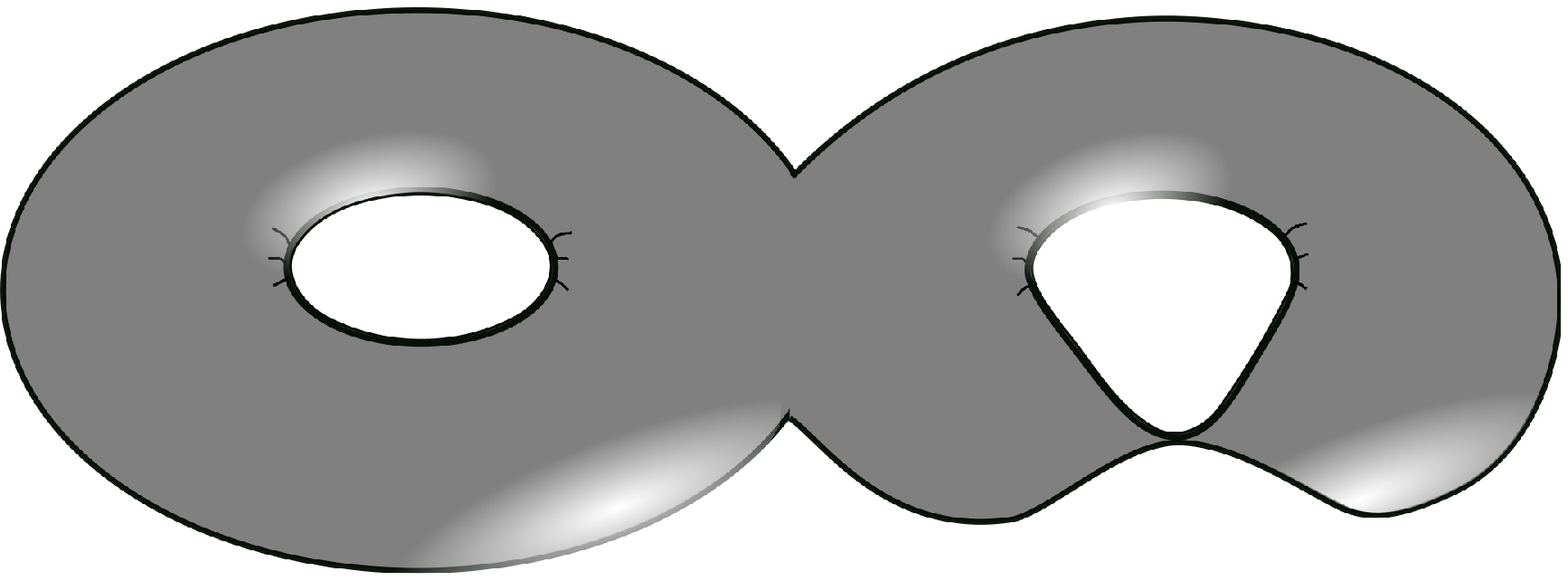}}
\end{minipage}
\begin{minipage}{6cm}
\center{\includegraphics[scale=0.3]{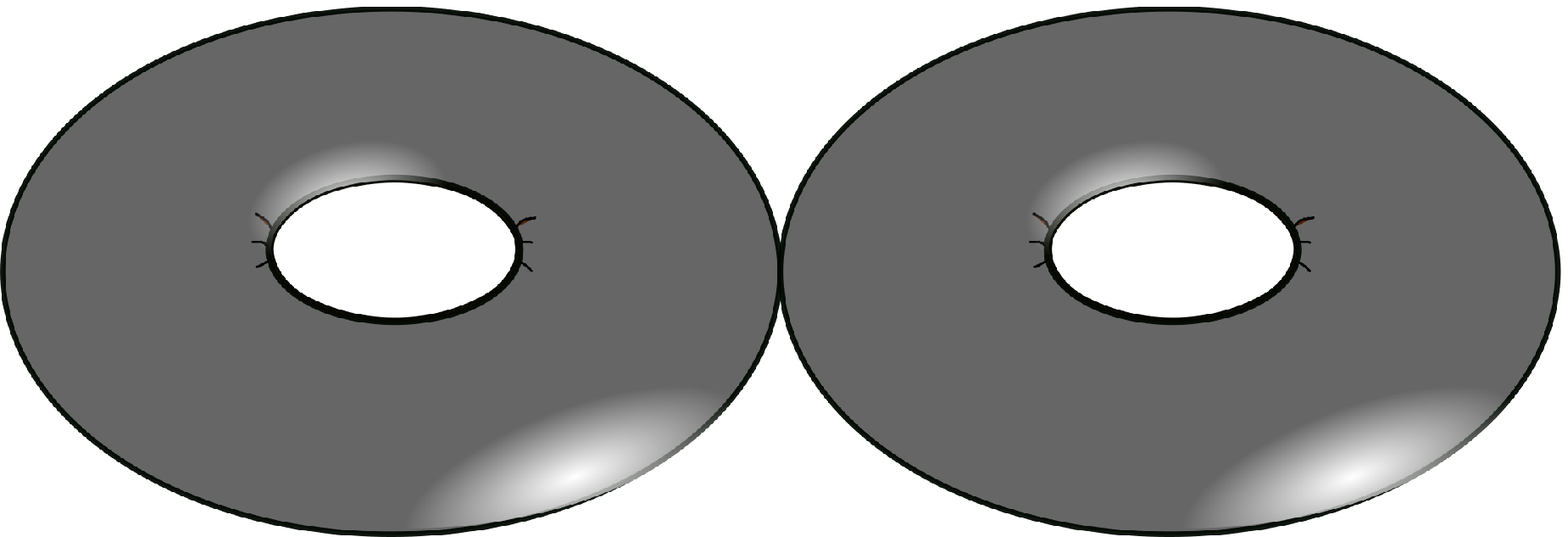}}
\end{minipage}
\begin{center}
\caption{Aspherical two-complexes.}\label{pinch}
\end{center}
\end{figure}
\end{ej}

\begin{obs}
It is well-known that the fundamental group of any $2$-dimensional aspherical complex is torsion-free (see  \cite[Proposition 2.45]{Hat}).  Theorem \ref{main5} says that if the $2$-complex $K$ has a torsion-free fundamental group and the maps  $i_*:\pi_1 (M)\to \pi_1 (K')$ are non-trivial then $K$ is aspherical.  
\end{obs}

We derive from Theorem \ref{main5} a result on asphericity of group presentations. This result resembles in some sense the homological description of $\pi_2$ using Reidemeister chains \cite[Thm 3.8]{Si} (See also \cite{Bo}). Given a group presentation $P$, let $K_P$ be the usual two-dimensional CW-complex associated to the presentation, which has one $0$-cell, one $1$-cell for each generator and one $2$-cell for each relator. The presentation $P$ is called aspherical if $K_P$ is aspherical. In order to study asphericity of $P$,  we will construct a digraph $D_P$ associated to $P$ together with a $G$-coloring. First note that the notion of a $G$-coloring naturally extends to directed graphs. A $G$-coloring of a digraph $D$ is a labeling of the edges of $D$ by elements in $G$. We allow loops and parallel edges which could have different colors. The color of the inverse of an edge $e$ is the inverse $c(e)^{-1}$ of the color of $e$.
A $G$-coloring $c$ induces a weight map $w_c$. If $\alpha = e_0e_1\ldots e_n$ is a cycle in the underlying undirected graph of $D$ (for each $i$, $e_i$ is an edge of $D$ or $e_i^{-1}$ is an edge of $D$), then $w_c(\alpha)=c(e_0)c(e_1)\ldots c(e_n)$.


Let $P=<a_1,a_2, \ldots , a_k \ | \ r_1, r_2, \ldots r_s>$ be a presentation of a group $G$. The vertices of the directed graph $D_P$ are the letters $a_i$ which appear in total exactly twice in the words
$r_1, r_2, \ldots ,r_s$. So, $a_i$ appears either with exponent $2$ or $-2$ in one of the relators and does not appear in any other relator, or it appears twice (in the same relator or in two different relators) with exponent $1$ or $-1$ each time.
Each vertex of $D_P$ will be the source of exactly two oriented edges and the target of two directed edges. Let $r=r_j=a_{i_0}^{\epsilon_0}a_{i_1}^{\epsilon_1} \ldots a_{i_{t-1}}^{\epsilon_{t-1}}$ be one of the relators of $P$, $\epsilon_l=\pm 1$ for every $l\in \mathbb{Z}_t$.
We consider $r$ as a cyclic word, so for example $a_{i_1}$ comes after $a_{i_0}$ and $a_{i_0}$ comes after $a_{i_{t-1}}$. Suppose $a_{i_l}$ is a vertex of $D_P$. We consider the first letter $a_{i_{l+m}}$ coming after
$a_{i_l}$ which is a vertex of $D_P$ (i.e. the minimum $m>0$ such that $a_{i_{l+m}} \in D_P$). It could be a letter different from $a_{i_l}$ or the same letter if $a_{i_l}$ appears twice in $r$ or if it appears once and no other $a_{i_s}$ is a vertex of $D_P$. Then $(a_{i_l}, a_{i_{l+m}})$ is
a directed edge of $D_P$ and the color corresponding to that edge is the subword $ga_{i_{l+1}}^{\epsilon_{l+1}}a_{i_{l+2}}^{\epsilon_{l+2}} \ldots a_{i_{l+m-1}}^{\epsilon_{l+m-1}}h\in G$ where $g=1$ if $\epsilon _{l}=1$ and $g=a_{i_{l}}^{\epsilon_{l}}$ if $\epsilon_l=-1$,
$h=1$ if $\epsilon _{l+m}=-1$ and $h=a_{i_{l+m}}^{\epsilon_{l+m}}$ if $\epsilon_{l+m}=1$.

The next example illustrates the situation.

\begin{ej} \label{exocto}
Figure \ref{grafou} shows the digraph $D_P$ corresponding to the presentation $P=$ $<a,b,c,d,e \ | \ b^2ca^{-1}b^{-1}dba, c^{-1}debe>$. Its vertices are $a,c,d$ and $e$. 

\begin{figure}[h] 
\begin{center}
\includegraphics[scale=0.6]{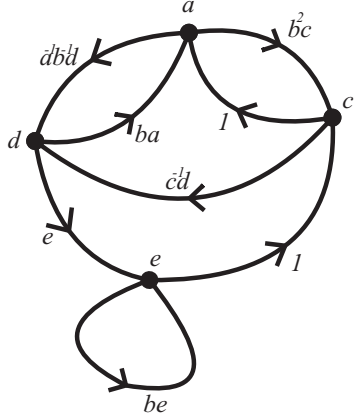}
\caption{The digraph $D_P$ associated to $P$.}\label{grafou}
\end{center}
\end{figure} 
\end{ej}

\begin{teo} \label{main6}
Let $P$ be a presentation of a group $G$. Suppose that every relator in $P$ contains a letter which is a vertex of $D_P$. If each component of $D_P$ contains a cycle whose weight has infinite order in $G$, then $P$ is aspherical. 
\end{teo}
\begin{proof}
We subdivide $K_P$ barycentrically to obtain a regular CW-complex $K$ as usual. Each $1$-cell corresponding to a generator $a$ in $P$ is subdivided in two $1$-cells $e_{a_0}$ and $e_{a_1}$ sharing the unique vertex $v$ of $K_P$ and a new vertex $v_a$. The $2$-cell $f_r$ corresponding to a relator $r$ of $P$ is subdivided in $2m$ $2$-cells where $m$ is the number of letters in $r$, adding a new $0$-cell $v_r$ in the interior of the original $2$-cell. Let $L$ be the $1$-dimensional subcomplex of $K'$ defined as in the statement of Theorem \ref{main5}. The vertices of $L$ are the barycenters of the $2$-cells of $K$ and the barycenters of the $1$-cells which are faces of exactly two $2$-cells. In the interior of the cell $f_r$ there are exactly $4m$ vertices of $L$ (the barycenters of the $2m$ $2$-cells and the barycenters of the $2m$ edges from $v_r$ to $v$ and to each $v_a$). This $1$-dimensional complex of $4m$ vertices is a cycle that we denote $C_r$.
The remaining vertices of $L$ are the barycenters $b(e_{a_0})$ and $b(e_{a_1})$ for each letter $a$ which is a vertex of $D_P$. We show that the hypotheses of the theorem ensure that the hypotheses of Theorem \ref{main5} are fulfilled.

\begin{figure}[h] 
\begin{center}
\includegraphics[scale=0.65]{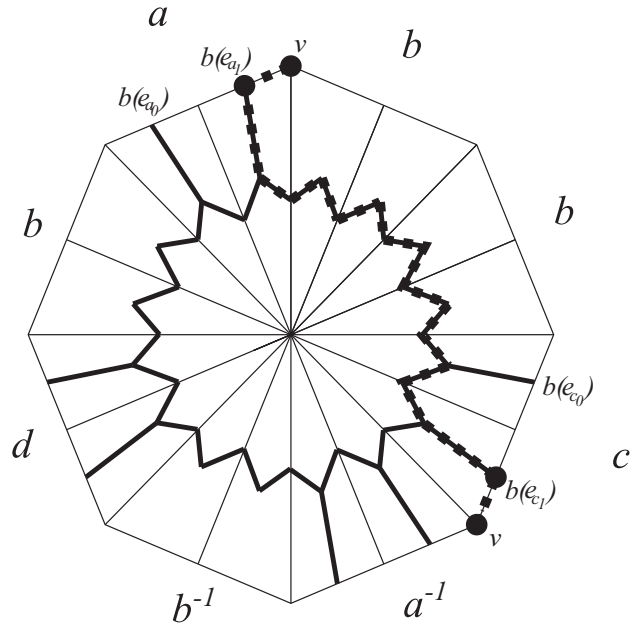}
\caption{The $2$-cell $f_r$. Here $r$ is the first relator of $P$ in Example \ref{exocto}. The simplices of the subcomplex $L$ are drawn with thick lines and the edge-path $\gamma _i$  homotopic to $b^2c\in G$ appears with dotted segments.}\label{octo}
\end{center}
\end{figure}

Since each relator contains a letter which is a vertex of $D_P$, the components of $D_P$ are in bijection with the components of $L$. Suppose $a$ and $c$ are vertices of $D_P$ and that there is an edge $(a,c)\in D_P$ (or $(c,a)$). Then, there is a relator $r$ of $P$ such that $a$ and $c$ are letters of $r$. Since $a,c \in D_P$, $b(e_{a_1})$ and $b(e_{c_1})$ are vertices of $L$ and they lie in the $2$-cell $f_r$ of $K_P$ corresponding to $r$. Moreover, there is an edge in $L$ from $b(e_{a_1})$ to the cycle $C_r$ and an edge from $b(e_{c_1})$ to $C_r$. Therefore there is and edge-path in $L$ from $b(e_{a_1})$ to $b(e_{c_1})$ entirely contained in $f_r$ (see Figure \ref{octo}). A cycle $\alpha$ in $D_P$ with base point $a$, has associated then a closed edge-path $\xi$ in $L$ at $b(e_{a_1})$. We will show that the order of $\xi$ in the edge-path group $E (K', b(e_{a_1}))$ is infinite or, equivalently, that the order of $\hat{\xi}=(v, b(e_{a_1}))\xi (b(e_{a_1}),v) \in E (K',v)$ is infinite. The edge-path $\hat{\xi}'$ obtained from 
$\hat{\xi}$ by inserting the edge-paths $(b(e_{l_1}),v)(v,b(e_{l_1}))$ at each vertex $b(e_{l_1})$ ($l$ a letter in $\alpha$) is equivalent to $\hat{\xi}$. Suppose $a=l^0, l^1, \ldots , l^k=a$ are the vertices of $\alpha$.  The edge-path $\hat{\xi}'$ is a composition of closed edge-paths $\gamma _i$ in $K'$ at $v$, each of them contained in a $2$-cell $f_{r_i}$.  The edge-path $\gamma _i$, as an element of $\pi _1 (K,v)$, is homotopic to a loop contained in the boundary of $f_{r_i}$ which is, as an element of $G$, the color of the edge $(l^i,l^{i+1})$ in $\alpha$. Thus, $\hat{\xi}'\in \pi _1 (K,v)\simeq G$ coincides with the weight of $\alpha$ and the first one has infinite order provided the second one does.   
   
\end{proof}

In Example \ref{exocto} there is an edge from $c$ to $d$ with color $c^{-1}d$, an edge from $a$ to $d$ with color $a^{-1}b^{-1}d$ and an edge from $a$ to $c$ with color $b^2c$. Therefore, there is a cycle with base point $c$ whose weight is $c^{-1}d(a^{-1}b^{-1}d)^{-1}b^2c=c^{-1} bab^2c\in G$. It is easy to verify that this element has infinite order, since $a+3b$ clearly has infinite order in the abelianization $G/[G:G]$. Since $D_P$ has a unique component and both relators of $P$ have at least one letter in $D_P$, Theorem \ref{main6} applies. This shows that $P$ is aspherical.

\end{document}